\documentclass[12pt,reqno]{amsart}

\usepackage[T1]{fontenc}
\usepackage{ae,aecompl}
\usepackage{graphics,cancel,graphicx}
\usepackage{epsfig,psfrag,manfnt}
\usepackage{mathrsfs}
\usepackage{graphicx,mathabx}
\usepackage{bbm,subfigure,rotating,bm,float,mathdots,wasysym,scalerel}
\newcommand\bigzero{\makebox(0,0){{\scaleobj{2.1}{0}}}}
\newcommand\Zero{\makebox(0,-3){{\scaleobj{1.5}{0}}}}
\newcommand\Zerob{\makebox(6,15){{\scaleobj{1.5}{0}}}}
\newcommand*{\bord}{\multicolumn{1}{c|}{}}
\usepackage[all,cmtip]{xy}
\usepackage{amssymb,amsmath,amsfonts,amsthm,color}

\usepackage{scalerel,stackengine}
\stackMath
\newcommand\reallywidehat[1]{%
\savestack{\tmpbox}{\stretchto{%
  \scaleto{%
    \scalerel*[\widthof{\ensuremath{#1}}]{\kern-.6pt\bigwedge\kern-.6pt}%
    {\rule[-\textheight/2]{1ex}{\textheight}}
  }{\textheight}%
}{0.5ex}}%
\stackon[1pt]{#1}{\tmpbox}%
}

\newcommand{\calA}{\mathcal{A}}

\newcommand{\calL}{\mathcal{L}}

\newcommand{\calN}{\mathcal{N}}

\newcommand{\mC}{\mathbb{C}}

\newcommand{\mN}{\mathbb{N}}

\newcommand{\bbv}{\mathbf{v}}

\newcommand{\balpha}{\bm{\alpha}}
\newcommand{\bbeta}{\bm{\beta}}
\newcommand{\bgamma}{\bm{\gamma}}
\newcommand{\bxi}{\bm{\xi}}

\newtheorem{theorem}{Theorem}[section]

\newtheorem{corollary}[theorem]{Corollary}
\newtheorem{proposition}[theorem]{Proposition}

\theoremstyle{definition}

\theoremstyle{definition}

\theoremstyle{definition}

\theoremstyle{definition}

\begin{document}



 \title[]{$\textrm{THE SYLVESTER EQUATION IN BANACH ALGEBRAS}$}
 
\vspace{-0.3cm}
 
 \author[]{Amol Sasane}
\address{Department of Mathematics \\London School of Economics\\
     Houghton Street\\ London WC2A 2AE\\ United Kingdom}
\email{A.J.Sasane@lse.ac.uk}
 
 \maketitle
 
  \vspace{-0.72cm}
 
 \begin{abstract}
 Let $\calA$ be a unital complex semisimple Banach algebra, and $M_{\calA}$ denote its maximal ideal space. For a matrix $M\in \calA^{n\times n}$,  $\widehat{M}$ denotes the matrix obtained by taking entry-wise Gelfand transforms. For a matrix $M\in \mC^{n\times n}$, 
 $\sigma(M)\subset \mC$ denotes the set of eigenvalues of $M$. 
 It is shown that if $A\in \calA^{n\times n}$ and $B\in \calA^{m\times m}$ are such that for all $\varphi \in M_{\calA}$, 
 $\sigma(\widehat{A}(\varphi))\cap \sigma(\widehat{B}(\varphi))=\emptyset$, then for all 
 $C\in \calA^{n\times m}$, the Sylvester equation $AX-XB=C$ has a unique solution $X\in \calA^{n\times m}$. As an application,  Roth's removal rule is proved in the context of matrices over a  Banach algebra.
 \end{abstract}

 \section{Introduction}
 
 \noindent The following result due to Sylvester is classical (see \cite{Syl}, \cite{BR}). Here for an $M\in \mC^{n\times n}$, $\sigma(M):=\{\lambda\in \mC: \lambda \textrm{ is an eigenvalue of }M\}$. 
 
 \begin{proposition}
 \label{t_0.1} 
 Let $A\in \mC^{n\times n}$ and $B\in \mC^{m\times m}$.
 For any $C\in \mC^{n\times m}$ the Sylvester equation 
  $
 AX-XB=C
 $  has a unique solution $X\in \mC^{n\times m}$ if and only if 
  $
 \sigma(A)\cap \sigma(B)=\emptyset.
 $  
 \end{proposition}
 
 \noindent The aim in this article is to prove an appropriate generalisation of this result 
 when $\mC$ is replaced by a commutative unital complex semisimple Banach algebra $\calA$. 
 For background on the Gelfand transform and spectral theory of Banach algebras, we refer the interested reader to e.g. \cite{Mul}, \cite[Chap. I]{Gam} or \cite[Part III, Chap. 11]{Rud}. 
 Let $M_{\calA}$ denote the maximal ideal space of $\calA$, consisting of all complex homomorphisms $\varphi:\calA\rightarrow \mC$. The dual space $\calL(\calA,\mC)$ of $\calA$ is equipped with the weak-$\ast$ topology, and $M_\calA\subset \calL(\calA,\mC)$ is given  the subspace topology induced from $\calL(\calA,\mC)$. Then $M_\calA$ is a compact Hausdorff topological space. The space of all complex-valued continuous functions on $M_\calA$ will be denoted by $C(M_\calA)$. For $x\in \calA$, the Gelfand transform of $x$, namely the map $M_\calA\owns \varphi\mapsto \varphi(x)$, will be denoted by $\widehat{x}\in C(M_\calA)$.  Let $\calA^{n\times m}$  denote the set of all $n\times m$ matrices with entries from $\calA$. For a matrix $X=[x_{ij}]\in \calA^{n\times m}$, we denote by 
 $\widehat{X}=[\widehat{x}_{ij}]\in C(M_{\calA})^{n\times m}$ the matrix of Gelfand transforms $\widehat{x}_{ij}$ 
 of the entries $x_{ij}$ of $X$. Our main result is the following. 
 
 \begin{theorem}
 \label{c_0.1}
Let $\calA$ be a commutative unital complex semisimple Banach algebra. 
Let $A\in \calA^{n\times n}$ and $B\in \calA^{m\times m}$be such that 
$$
\forall \varphi \in M_{\calA}, \; \sigma(\widehat{A}(\varphi))\cap \sigma(\widehat{B}(\varphi))=\emptyset.
$$ 
 Then for every $C\in \calA^{n\times m},$  there exists a unique $X\in \calA^{n\times m}$ such that $AX-XB=C$. 
 \end{theorem}
 
 \subsection{Relation to previous and recent work}
 \label{subsection_known_results}
 
 The Sylvester equation $ax-xb=c$ has been studied in arbitrary Banach algebras in \cite{Ros}. However, our result is not a consequence of this result, because $A$ and $B$ are not of the same dimensions. Moreover, the result in \cite{Ros} gives a solvability condition in 
 terms of the Dunford-Taylor operational calculus, while we give a pointwise criterion using Gelfand transforms.
 We also mention some more recent papers on the topic of Sylvester operator equations. 
 It was pointed out by the reviewer that our proof of (the classical, known) Proposition~\ref{t_0.1} contains the same expressions as those derived in \cite{Dra}. Papers \cite{Djo} and \cite{Djo2} study the case when the Sylvester operator $S: X\mapsto AX-XB$  is not invertible, but the initial equation is still solvable (with infinitely many solutions). In particular,  \cite{Djo2} covers the case when $A$, $B$ and $C$ are scalar matrices while \cite{Djo} covers the case when $A$, $B$ and $C$ are bounded linear operators in Banach spaces. The results in \cite{Djo} are also obtained via the Gelfand transform and spectral theory for commutative unital Banach algebras.
 
 \smallskip 
 
 \noindent  The outline  of the article is as follows. 

\noindent $\bullet$  We collect some preliminaries in Section~\ref{section_preliminaries}. In particular, we repeat  
\phantom{$\bullet$ }the proof of the classical result because its proof will reveal that the \phantom{$\bullet$ }solution $X$ depends continuously on the data $A,B,C$, a fact which \phantom{$\bullet$ }we will need to prove our Theorem~\ref{c_0.1}. 
 We will also recall the Implicit 
 \phantom{$\bullet$ }Function Theorem in Banach algebras, which will be our main tool. 
 
\noindent $\bullet$ Subsequently, in Section~\ref{section_proof}, we will 
 give the proof of Theorem~\ref{c_0.1}. 
 
\noindent $\bullet$ Finally, in Section~\ref{section_application}, we give an application of our main result to prove  \phantom{$\bullet$ }an analogue of the Roth removal rule (a criterion for the similarity  \phantom{$\bullet$ }of a block diagonal matrix and a block upper triangular matrix) in  \phantom{$\bullet$ }the context of matrices over a commutative Banach algebra. 

\smallskip 

\noindent {\bf Acknowledgement:} I am grateful to the anonymous referee for the careful review, and for useful comments. In particular, for drawing my attention to  some of the references to recent results on the operator Sylvester equation that are now cited in Subsection~\ref{subsection_known_results}.
 
 \goodbreak 
 
 \section{Preliminaries}
 \label{section_preliminaries}
 
 \noindent We use the notation $I_n$ for the $n\times n$ identity matrix. 
 
 \subsection{Proof when $\calA=\mC$}
 
  \begin{proposition}
 \label{t_0.1}
 Let $A\in \mC^{n\times n}$ and $B\in \mC^{m\times m}$ be such that 
 $$
 \sigma(A)\cap \sigma(B)=\emptyset.
 $$ 
 Then for every $C\in \mC^{n\times m}$ there exists a unique $X\in \mC^{n\times m}$ such that $AX-XB=C$.
 \end{proposition}
 \begin{proof} Let $L:\mC^{n\times m}\rightarrow \mC^{n\times m}$  be the linear transformation given by  $L(X)=AX-XB$ for all $X\in \mC^{n\times m}$. We want to show that $L$ is invertible. It is enough to show it is injective. 
 Let $X\in \ker L$, that is, $AX-XB=0$. Then $AX=XB$. It follows by induction that for all $k\geq 0$, $A^kX=XB^k$ (since if true for some $k$, then we have $A^{k+1}X=A(A^kX)=A(XB^k)=(AX)B^k=(XB) B^k=XB^{k+1}$). Suppose $p_A, p_B\in C[z]$ are the characteristic polynomials of $A,B$. 
 As $\sigma(A)\cap \sigma(B)=\emptyset$, it follows that $p_A,p_B$ are coprime. 
 So there exist polynomials $q,\widetilde{q}\in \mC[z]$ such that $q \;\!p_A +\widetilde{q} \;\!p_B=1$. 
 By the Cayley-Hamilton theorem, $p_A(A)=0$ and $p_B(B)=0$. We have 
 \begin{eqnarray*}
 0\!\!\!\!&=&\!\!\!\! 0\;\!X=q(A)\;\! 0 \;\!X=q(A) \;\!p_A(A) \;\!X=(I_n-\widetilde{q}(A) \;\!p_B(A))X
 \\
 \!\!\!\!&=&\!\!\!\!X-\widetilde{q}(A)\;\! p_B(A)\;\!X=X-\widetilde{q}(A) \;\!X\;\!p_B(B)
 \\
 \!\!\!\!&=&\!\!\!\!X-\widetilde{q}(A)\;\! X \;\!0 
 =X-0=X.
 \end{eqnarray*}
 So $L$ is injective, and hence invertible.  
 \end{proof}
 
 \noindent We endow $\mC^{n\times n}$ with the operator norm topology induced by equipping $\mC^{n}$ with the topology given by the Euclidean $2$-norm $\|\cdot\|_2$:
 $$
 \|\bbv\|_2:=\sqrt{v_1^2+\cdots+v_n^2}\;\; \textrm{ for }\;\bbv\!=\!\left[\!\!\begin{array}{c}v_1\\ \vdots \\ v_n\end{array}\!\!\right]\in \mC^{n}.
 $$
 Thus if $M\in \mC^{n\times n}$, then 
 $\displaystyle 
 \|M\|=\sup_{\mathbf{0}\neq \bbv\in \mC^n}\frac{\|M\bbv\|_2}{\|\bbv\|_2}.
 $ 
 
 \begin{corollary} 
 \label{corollary_15_july_2021_12:41}
 Let $A_0\in \mC^{n\times n}$ and $B_0\in \mC^{m\times m}$ be such that 
 $$
 \sigma(A_0)\cap \sigma(B_0)=\emptyset.
 $$ 
 Then there exist neighbourhoods $\calN_{A_0},\calN_{B_0}$ of $A_0,$ respectively $B_0,$  such that for all 
 $(A,B)\in \calN_{A_0}\times \calN_{B_0}$, we have 
 $$
 \sigma(A)\cap \sigma(B)=\emptyset.
 $$ 
 For $(A,B)\in \calN_{A_0}\times \calN_{B_0}$ and $C\in \mC^{n\times m},$ let $X(A,B,C)$ denote the unique solution $X\in \mC^{n\times m}$ to 
 $AX-XB=C$. Then the map 
 $$
\calN_{A_0}\times \calN_{B_0}\times \mC^{n\times n} \owns (A,B,C)\mapsto X(A,B,C)\in \mC^{n\times n}
 $$ is continuous. 
 \end{corollary}
 \begin{proof} It is clear that the coefficients of the characteristic polynomial of a matrix depend continuously on the matrix. Also, the roots of a polynomial depend continuously on its coefficients
 (see e.g. \cite{Hir} for a precise statement and a proof). Thus the eigenvalues of a matrix, being the roots of the characteristic polynomial, depend continuously on the matrix. 
 
 As the spectrum $\sigma(M)$ of a matrix $M\in \mC^{n\times n}$ is a finite set comprising at most $n$ distinct complex numbers,  given $A_0\in \mC^{n\times n}$ and $B_0\in \mC^{m\times m}$ such that $\sigma(A_0)\cap \sigma(B_0)=\emptyset$, there exist neighbourhoods $D_{A_0}$ and $D_{B_0}$ of $\sigma(A_0)$, respectively of $\sigma(B_0)$, in $\mC$ 
 such that $D_{A_0}\cap D_{B_0}=\emptyset$ (because  the Euclidean topology of the complex plane is Hausdorff).  By the continuity of eigenvalues mentioned in the first paragraph above, it follows that there is a neighbourhood $\calN_{A_0}$ of $A_0$ and a neighbourhood $\calN_{B_0}$ of $B_0$  such that for all 
 $A,B\in  \calN_{A_0}\times \calN_{B_0}$, we have $\sigma(A)\subset D_{A_0}$ and $\sigma(B)\subset D_{B_0}$, so that in particular, $\sigma(A)\cap \sigma(B)=\emptyset$. 
 
  The map $L_{A,B}\in \calL(\mC^{n\times m},\mC^{n\times m})$, given by 
  $$
  L_{A,B}(X)=AX-XB\textrm{ 
  for all }X\in \mC^{n\times m},
  $$
   depends continuously on $A,B$. Indeed, using the property of the operator norm that 
    $\|PQ\|\leq \|P\|\|Q\|$ (for complex matrices $P,Q$),  we get 
    $$
  \|L_{A,B}-L_{A_0,B_0}\|\leq \|A-A_0\|+\|B-B_0\|.
  $$ 
   We also know that $L$ is invertible whenever $\sigma(A)\cap \sigma(B)=\emptyset$ (from Theorem~\ref{t_0.1}).  Let $GL_{nm}(\mC)$ denote the invertible maps in the set $ \calL(\mC^{n\times m},\mC^{n\times m})$. Since the operation of taking inverse, namely the map $\cdot^{-1}: GL_{nm}(\mC)\rightarrow GL_{nm}(\mC)$, is continuous,  we have that 
$$
 \calN_{A_0}\times \calN_{B_0}\times \mC^{n\times n}\owns (A,B,C)\mapsto X(A,B,C)=(L_{A,B})^{-1} C\in \mC^{n\times m}
 $$
 is a continuous map. 
 \end{proof}

 \subsection{The Implicit Function Theorem in Banach algebras}
 
\noindent  The main tool  we will use  is the following 
Implicit Function Theorem in Banach Algebras (see \cite[p.155]{Hay}). 
This will afford us passage  from continuous functions on $M_\calA$ to
elements of $\calA$.

\begin{proposition}
\label{prop3}
Let $\calA$ be a commutative unital complex semisimple Banach algebra.  
Let $h_1,\cdots, h_s$ be continuous functions on $M_\calA$. Suppose that
$f_1,\cdots f_\ell$ in $\calA$ and $
G_1(z_1,\cdots, z_{s+\ell}),\dots, G_t(z_1,\cdots,z_{s+\ell})
$ are holomorphic functions 
with $t\geq s$ defined on a neighbourhood of the joint spectrum
$$
\begin{array}{lr}
\sigma(h_1,\cdots, h_s, f_1,\cdots, f_\ell)\phantom{AAAA}
\\\phantom{AAAAAA}:=\{ (h_1(\varphi),\cdots,h_s(\varphi), \widehat{f_1}(\varphi),\cdots,
\widehat{f_\ell}(\varphi)):\varphi\in M_\calA\},
\end{array}
$$
such that 
\begin{equation}
\label{eq_hayashi_cont_soln}
G_k(h_1,\cdots, h_s, \widehat{f_1},\cdots,\widehat{f_\ell})=0 
\textrm{ on }M_\calA \textrm{ for }1\leq k\leq t.
\end{equation}
If the rank of the Jacobi matrix 
$$
\displaystyle
\frac{\partial(G_1,\cdots, G_t)}{\partial(z_1,\cdots, z_s)}
$$ 
is $s$ on
$\sigma(h_1,\cdots, h_s, f_1,\cdots, f_\ell)$, then there exist elements
$g_1,\cdots, g_s$ in $\calA$ such that
$$
\widehat{g_1}=h_1,\;\cdots,\;\widehat{g_s}=h_s .
$$
\end{proposition}

\goodbreak 
 
 \section{Proof of the main result}
 \label{section_proof}
 
\begin{proof}(Of Theorem~\ref{c_0.1}).  The condition 
 $\sigma(\widehat{A}(\varphi))\cap \sigma(\widehat{B}(\varphi))=\emptyset$ 
  for all $\varphi \in M_{\calA}$,  
implies (by Proposition~\ref{t_0.1}) the existence of a pointwise solution $F$, 
$M_\calA \owns \varphi \mapsto F(\varphi)\in \mC^{n\times m}$, satisfying 
$$
\widehat{A}(\varphi) F(\varphi) -F(\varphi) \widehat{B}(\varphi)=\widehat{C}(\varphi) 
\textrm{ for all }\varphi \in M_{\calA}. \quad \quad (\star)
$$
We want to produce an $X\in \calA^{n\times m}$ such that $\widehat{X}=F$. We note that in this case,  as $\widehat{X}\in C(M_{\calA})^{n\times m}$, we should have $F$ depend continuously on $\varphi$. Corollary~\ref{corollary_15_july_2021_12:41} implies for any $\varphi_0\in M_\calA$, 
there exists a neighbourhood $U\subset M_\calA$ of $\varphi_0$ such that (the unique pointwise solution) $F|_{U}\in C(U)^{n\times m}$. It follows that $F\in C(M_\calA)^{n\times m}$. 
  
  Now we will prove that $\widehat{X}=F$ by using the Banach algebra Implicit Function Theorem, namely Proposition~\ref{prop3}.  In our case, 
$s=nm$, $t=nm$, the $h_i$ ($1\leq i\leq nm$)  are the $nm$ component functions of $F$, and the $f_i$ ($1\leq i\leq \ell=n^2+m^2+nm$)  comprise the components of $A,B, C$ (which
are totally $\ell=n^2+m^2+nm$ in number).  The maps
$G_1,\dots G_{t=nm}$ are the $nm$ components of the map
$$
\mC^{n\times n}\times 
\mC^{m\times m}\times
\mC^{n\times m}\times
\mC^{n\times m}\owns (\balpha,\bbeta,\bgamma,\bxi)\mapsto \balpha \bxi-\bxi \bbeta -\bgamma\in \mC^{n\times m}.
$$
(In the above, we have the replacements of $A,B,C$ by the
complex variables which are the components of $\balpha,\bbeta,\bgamma$, respectively.
The replacement of the $X$ in the Sylvester equation is by the complex
variables which are the components of $\bxi$.)  Clearly, the above
map is holomorphic not just on a neighbourhood of the joint spectrum, but in fact in the
whole of $\mC^{s+\ell}=\mC^{nm+ n^2+m^2+nm}$. Also, the condition \eqref{eq_hayashi_cont_soln} in
Proposition~\ref{prop3} is satisfied, because  $F\in C(M_\calA)^{n\times m}$ satisfies 
 ($\star$) above. 
 
 So we now investigate the Jacobian with respect to the variables in
$\bxi$. The Jacobian with respect to the $\bxi$ variables at the
point
$$
(\!F(\varphi), \;\widehat{A}(\varphi), \;\widehat{B}(\varphi),\;
\widehat{C}(\varphi))\in \sigma(h_1,\cdots, h_{nm}, f_1,\cdots, f_{n^2+m^2+nm})
$$
is the  linear transformation $\bxi \stackrel{\Lambda}{\mapsto} \widehat{A}(\varphi)\;\!\bxi -\bxi \;\!\widehat{B}(\varphi):\mC^{n\times m}\rightarrow \mC^{n\times m}$. This map $\Lambda$ is invertible, thanks to the condition 
 $
\sigma(\widehat{A}(\varphi))\cap \sigma(\widehat{B}(\varphi))=\emptyset.
$ 
 So the rank of the Jacobian with respect to the variables in
$\bxi$ is $nm=s$ on the joint spectrum. Hence $F=\widehat{X}$ for some $X\in \calA^{n\times m}$. 

\smallskip 

\noindent {\em Uniqueness}: Suppose $X,Y$ are two solutions such that $X\neq Y$. As the Banach algebra $\calA$ is semisimple, there exists a $\varphi_0\in M_\calA$ such that $\widehat{X}(\varphi_0)\neq \widehat{Y}(\varphi_0)$. But then we get two solutions $\widehat{X}(\varphi_0)\neq \widehat{Y}(\varphi_0)\in \mC^{n\times n}$ to the Sylvester equation 
$$
\widehat{A}(\varphi_0) \;\!\bxi - \bxi\;\! \widehat{B}(\varphi_0)=\widehat{C}(\varphi_0)
$$
despite $\widehat{A}(\varphi_0) \cap \widehat{B}(\varphi_0)=\emptyset$, contradicting Proposition~\ref{t_0.1}. 
\end{proof}

\goodbreak 

 \section{$\textrm{Application:  Roth's removal rule}$} 
 \label{section_application}
 
 \noindent The following result was proved in \cite{Rot}.
 
 \begin{proposition}
 \label{prop_17_jul_2021_14:07}
 Let $A\in \mC^{n\times n},$ $B\in \mC^{m\times m}$ and $C\in \mC^{n\times m}$. Then  
 $$
 \left[\begin{array}{cc} A& 0\\
 0& B\end{array}\right]
 \textrm{ and }
 \left[\begin{array}{cc} A& C\\
 0& B\end{array}\right] 
 $$ 
 in $\mC^{(n+m)\times (n+m)}$ are similar if and only if there exists an $X\in \mC^{n\times m}$ such that $AX-XB=C$. 
 \end{proposition}
 
 \noindent For an arbitrary unital commutative ring $R$, we note that if 
 $A\in R^{n\times n}$, $B\in R^{m\times m}$, $C\in R^{n\times m}$, and there 
exists an $X\in R^{n\times m}$ such that $AX-XB=C$, then setting 
$$
S:=\left[\begin{array}{cc} I_n & X\\
 0& I_m\end{array}\right]\in R^{(n+m)\times (n+m)},
 $$
 we have 
 $$
 S^{-1}:=\left[\begin{array}{cc} I_n & -X\\
 0& I_m\end{array}\right]\in R^{(n+m)\times (n+m)},
 $$
 and so  
 \begin{eqnarray*}
S\left[\begin{array}{cc} A& C\\
 0& B\end{array}\right] S^{-1}
 \!\!\!\!&=&\!\!\!\!
 \left[\begin{array}{cc} I_n & X\\
 0& I_m\end{array}\right]\left[\begin{array}{cc} A& C\\
 0& B\end{array}\right] \left[\begin{array}{cc} I_n & -X\\
 0& I_m\end{array}\right]
 \\
 \!\!\!\!&=&\!\!\!\!
 \left[\begin{array}{cc} A & C+XB\\
 0& B\end{array}\right]
 \left[\begin{array}{cc} I_n & -X\\
 0& I_m\end{array}\right]
 \\
 \!\!\!\!&=&\!\!\!\!
 \left[\begin{array}{cc} A & C+XB-AX\\
 0& B\end{array}\right]\\
 \!\!\!\!&=&\!\!\!\!  \left[\begin{array}{cc} A & C-C\\
 0& B\end{array}\right]
 \\
 \!\!\!\!&=&\!\!\!\!\left[\begin{array}{cc} A & 0\\
 0& B\end{array}\right].
 \end{eqnarray*}
In fact the converse is also true, and Proposition~\ref{prop_17_jul_2021_14:07} can be generalised to the case of arbitrary rings \cite{Gus}.
 
 \begin{proposition}
 \label{prop_gustafsson}
 Let $R$ be a commutative unital ring. 
 Let $A\in R^{n\times n},$ $B\in R^{m\times m}$ and $C\in R^{n\times m}$. The matrices 
 $$
 \left[\begin{array}{cc} A& 0\\
 0& B\end{array}\right]
 \textrm{ and }
 \left[\begin{array}{cc} A& C\\
 0& B\end{array}\right] 
 $$ 
 in $R^{(n+m)\times (n+m)}$ are similar if and only if there exists an $X\in R^{n\times m}$ such that $AX-XB=C$. 
 \end{proposition}
 
\noindent We have the following consequence of our main result.

 \begin{corollary}
 Let $\calA$ be a commutative unital complex semisimple Banach algebra. 
 Let $A\in \calA^{n\times n},$ $B\in \calA^{m\times m},$  $C\in \calA^{n\times m}$. 
 Then the following are equivalent:
 
 \noindent {\em (1)} For every $C\in \calA^{n\times m},$ the matrices 
 $$
 \left[\begin{array}{cc} A& 0\\
 0& B\end{array}\right]
 \textrm{ and }
 \left[\begin{array}{cc} A& C\\
 0& B\end{array}\right] 
 $$ 
\phantom{(1) }in $\calA^{(n+m)\times (n+m)}$ are similar.
 
  \noindent {\em (2)}  For every $C\in \calA^{n\times m},$ there exists a unique $X\in \calA^{n\times m}$ such that \phantom{(1) }$AX-XB=C$.

 \noindent {\em (3)}  For all $\varphi \in M_{\calA},$ $ \sigma(\widehat{A}(\varphi))\cap \sigma(\widehat{B}(\varphi))=\emptyset. $
\end{corollary}
\begin{proof}  $\;$

\noindent (3)$\Rightarrow$(2) follows from Theorem~\ref{c_0.1}. 

\smallskip 

\noindent (2)$\Rightarrow$(1) follows from Proposition~\ref{prop_gustafsson}. 

\smallskip 

\noindent (1)$\Rightarrow$(3): From Proposition~\ref{prop_gustafsson}, it follows that  for every $C\in \calA^{n\times m}$, there exists an $X\in \calA^{n\times m}$ such that $AX-XB=C$. 
 Take any matrix $C_0\in \mC^{n\times m}$, and set $C=C_0 e$, where $e\in \calA$ is the unit element of the Banach algebra $\calA$. Then $\widehat{C}=C_0 {\mathbf{1}}_{M_\calA}$, where 
 ${\mathbf{1}}_{M_\calA}\in C(M_\calA)$ is the function identically equal to $1$ on $M_\calA$. Let  $\varphi \in M_\calA$. Then the matrix $X_0:=\widehat{X}(\varphi)$ satisfies 
 $$
 \widehat{A}(\varphi)X_0-X_0 \widehat{B}(\varphi)=C_0 .
 $$
 As $C_0\in  \mC^{n\times m}$ was arbitrary, 
  the map 
  $$
  L_{ \widehat{A}(\varphi), \widehat{B}(\varphi)}:  \mC^{n\times m}\rightarrow  \mC^{n\times m}, \quad 
   \mC^{n\times m}\owns Y\mapsto  \widehat{A}(\varphi)Y-Y\widehat{B}(\varphi)\in \mC^{n\times m}
  $$ 
  is surjective, and hence invertible. Hence
   $
  \sigma(\widehat{A}(\varphi))\cap \sigma(\widehat{B}(\varphi))=\emptyset.
  $ Also, since $\varphi\in M_\calA$ was arbitrary, we conclude that (3) holds. 
\end{proof}
 
 \noindent A repeated application of the previous result gives the following. 
 
 \begin{corollary}
 Let $\calA$ be a commutative unital complex semisimple Banach algebra. 
 Suppose $A_{ii}\in \calA^{d_i\times d_i},$ $d_i\in \mN,$ $i\in \{1,\cdots ,n\},$  
 satisfy 
 $$
 \forall \varphi \in M_\calA,\; 
 \sigma(\widehat{A}_{ii}(\varphi))\cap  
 \sigma(\widehat{A}_{jj}(\varphi))=\emptyset, \textrm{ for } 1\leq i<j\leq n.
 $$
 For $1\leq i<j\leq n$, let $A_{ij}\in\calA^{d_i\times d_j}$. Then 
 the matrices 
  $$
  \left[
    \begin{array}{cccc}
    A_{11}    & A_{12}       & \cdots    & A_{1n}    \\ \cline{1-1}
    \bord & A_{22}       & \cdots    & A_{2n}     \\ \cline{2-2}
          & \bigzero   & \ddots    &     \\ 
          &   & \bord & A_{nn}     \\ \cline{4-4}
  \end{array}\right]
 \textrm{ and }
  \left[\begin{array}{ccc}
  A_{11} &  & \Zero   \\
    & \ddots &    \\
 \Zerob  & & A_{nn} \end{array}\right]
 $$
 are similar in $\calA^{D\times D}$, where $D=d_1+\cdots +d_n$.
 \end{corollary}
%
%
%
%
%
%
%
%
%

\end{document}